\documentclass[a4paper, 12pt]{amsart}

\usepackage[utf8]{inputenc}
\usepackage{amsmath,amsfonts,amssymb,amsopn,amscd,amsthm}
\usepackage{comment}
\usepackage{dsfont}
\usepackage{graphicx}
\usepackage{color}
\usepackage[colorlinks]{hyperref}
\usepackage{epigraph}
\usepackage[left=3.5cm,top=2.5cm,bottom=2cm,right=3cm]{geometry}

\title[H\"ormander condition for delayed SDEs]
      {The H\"ormander condition for delayed stochastic differential equations}
\author{Reda   \textsc{Chhaibi}}
\address{Universit\'e Paul Sabatier, Toulouse 3 -- Institut de math\'ematiques de Toulouse (IMT) -- 118, route de Narbonne, 31400, Toulouse, France}
\email{reda.chhaibi@math.univ-toulouse.fr}
\date{\today}
\author{Ibrahim \textsc{Ekren}}
\address{Department of Mathematics -- Florida State University -- 1017 Academic Way,
Tallahassee, FL 32306, USA
}
\email{iekren@fsu.edu}
        
\allowdisplaybreaks[4]

\DeclareMathOperator{\id}{id}

\def\half{\frac{1}{2}}
\def\third{\frac{1}{3}}
\def\1{{\mathbf 1}}

\def\pa{\partial}

\def\a{\alpha}

\def\d{\delta}

\def\O{\Omega}

\def\ve{\varepsilon}

\def\N{{\mathbb N}}

\def\R{{\mathbb R}}

\def\W{\mathbb{W}}
\def\D{\mathbb{D}}
\def\P{{\mathbb P}}
\def\E{{\mathbb E}}

\def\X{{\mathbb X}}

\def\Cc{{\mathcal C}}
\def\Dc{{\mathcal D}}

\def\Fc{{\mathcal F}}

\def\Lc{{\mathcal L}}
\def\Mc{{\mathcal M}}
\def\Nc{{\mathcal N}}

\def\Pc{{\mathcal P}}

\def\Rc{{\mathcal R}}

\def\Vc{{\mathcal V}}
\def\Wc{{\mathcal W}}

\def\xb{{\mathbf x}}

\newtheorem{thm}{Theorem}[section]
\newtheorem{proposition}[thm]{Proposition}

\newtheorem{definition}[thm]{Definition}

\newtheorem{lemma}[thm]{Lemma}

\newtheorem{rmk}[thm]{Remark}
\newtheorem{assumption}[thm]{Assumption}

\numberwithin{equation}{thm}

\newcommand{\ba}{\begin{array}}
\newcommand{\ea}{\end{array}}
\newcommand{\be}{\begin{equation}}
\newcommand{\ee}{\end{equation}}
\newcommand{\bea}{\begin{eqnarray}}
\newcommand{\eea}{\end{eqnarray}}
\newcommand{\beaa}{\begin{eqnarray*}}
\newcommand{\eeaa}{\end{eqnarray*}}
\newcommand{\ignore}[1]{}
\newcommand{\vertiii}[1]{{\left\vert\kern-0.25ex\left\vert\kern-0.25ex\left\vert #1 
    \right\vert\kern-0.25ex\right\vert\kern-0.25ex\right\vert}}
\begin{document}

\begin{abstract}
In this paper, we are interested in path-dependent stochastic differential equations (SDEs) which are controlled by Brownian motion and its delays. Within this non-Markovian context, we give a H\"ormander-type criterion for the regularity of solutions. Indeed, our criterion is expressed as a spanning condition with brackets. A novelty in the case of delays is that noise can ``flow from the past'' and give additional smoothness thanks to semi-brackets.

The proof follows the general lines of Malliavin's probabilistic proof, in the Markovian case. Nevertheless, in order to handle the non-Markovian aspects of this problem and to treat anticipative integrals in a path-wise fashion, we heavily invoke rough path integration.
\end{abstract}

\maketitle

\medskip
\medskip
\medskip
\hrule
\tableofcontents
\hrule
\newpage

\section{Introduction}

This is a first paper on the general question of smoothness for marginals of solutions to non-Markovian SDEs. Here, we fix a time maturity $T>0$ and  $h = \left( h_i \right)_{0 \leq i \leq N-1} \in \R_+^N$ is an increasing sequence of delays satisfying:
\begin{align}
\label{eq:delays}
0 = & \ h_0 < h_1 < h_2 < \dots < h_{N-1} < T \ .
\end{align}
We also fix $m,d>0$ two integers and consider the random variable $X_T \in \R^d$ where $\left( X_t \right)_{0 \leq t \leq T}$ is the solution to a delayed SDE:
\begin{align}
\label{eq:sde_def}
X_t & = X_0 + \sum_{k=0}^m \int_0^t V_k(r,X)  \circ dW^k_r \ , \mbox{ for all }t\geq 0\\
X_t & = X_0 \ ,\mbox{ for all }t\leq 0 \ .
\end{align}
The process $(W^k)_{1\leq k\leq m}$ is an $m$-dimensional Brownian motion. The vector fields are of the form
$$V_k = V_k(t,X) = V_k\left( t, X_{t},X_{t-h_1},\ldots X_{t-h_{N-1}} \right) \in \R^d \ ,$$
and depend smoothly on delayed values of the path $X$. $N=1$ recovers the usual Markovian setting of diffusions. By convention, the additional index $k=0$ will refer to time and $W_t^0=t$. Notice that the Stratonovich stochastic integration $\circ dW^k_\cdot$ is a priori ill-defined as the integrand has no reason to be a semi-martingale. Recall that by definition, Stratononich integration is usual given in the setting of two continuous semi-martingales $(X,Y)$ by:
\begin{align*}
X_t \circ dY_t & = X_t dY_t + \half d \langle X, Y \rangle_t \ ,
\end{align*}
where $X_t dY_t$ is defined by Itô integration and $\langle X, Y \rangle_t$ is the covariation bracket which may not exist outside of the framework of semi-martingales.

At this point, we leave to later the discussion regarding which theory of stochastic integration is invoked. Here is a simple example showing that delays take us out of the usual framework: in the natural filtration $\Fc_t := \sigma\left( W_s ; s \leq t \right)$, $X_t = W_{t-h_1}$ itself cannot be a semi-martingale. Indeed, assume the semi-martingale decomposition $X_t = M_t + A_t$ such that $M$ is a local martingale and $A$ adapted with locally finite total variation. Upon localizing thanks to stopping times, we can assume that $M$ is a bounded martingale and that $A$ has bounded total variation. Now because of the delays, and Brownian motion having almost surely infinite total variation, necessarily $X$ has infinite mean variation
$$ MV_t := \sup_{0 < t_1 < t_2 < \dots < t}
           \left| \sum_k |\E \left[ X_{t_{k+1}} - X_{t_{k}}| \Fc_{t_k} \right]| \right|
         = \infty$$
for all $t>0$ and therefore the finite variation compensator $A$ cannot exist.

Let us now go back to discussing the topic of regularity. The Markovian setting i.e when $V_k = V_k(t,X_t)$ has a beautiful answer in the form of H\"ormander's spanning condition \cite{hormander1967}. In that paper, H\"ormander states that $X_T$ has a smooth density if certain Lie brackets between the vector fields $\left( V_k \right)_{0 \leq k \leq d}$ span $\R^d$. Of course, the language of H\"ormander was functional analysis and PDEs. The translation from probability to PDEs is readily obtained upon invoking the fact densities are the fundamental solutions to the forward Fokker-Planck PDE. Malliavin's proof pushed further by giving a probabilistic approach. We recommend \cite{Hairer11} for a comprehensive review.


Although everything can be recast into the Itô convention, we choose to work under the Stratonovich convention. This choice is not innocuous. Indeed, it is well-known that the Stratonovich reformulation in terms of vector fields is the right language for ``geometric'' arguments (see for e.g \cite{Hsu}). The H\"ormander condition itself is very geometric by nature, since it morally says that heat dissipates along the vector fields $V_k$ and their brackets, due to the erratic movement of Brownian motion. Another reason is the use of ``geometric rough paths'' (see \cite[Chapter 2.2]{FH14} for a definition and  \cite[Chapter 3]{FH14} for the discussion) thanks to which the Itô formula looks similar to the usual chain rule.

\medskip

{\bf Literature review: } In the literature, there are two ways of understanding the word ``non-Markovian'' regarding the topic of H\"ormander's hypoellipticity. On the one hand, certain authors as \cite{CT10, HP13, CHLT15} mean that the SDE's driving noise is a fractional Brownian motion or a general Gaussian process which may fail to be Markovian. On the other hand, another legitimate direction of investigation is to consider a source of non-Markovianity which is the path-dependence of the SDE. In this case, one needs to qualify the path-dependence, otherwise uniform ellipticity becomes the only recourse.

To the best of the authors' knowledge, regularity with a path-dependence via delays has been treated in the following papers. In \cite{BelMo91, BelMo95, stoica1998, Bell}, the authors deal with a martingale term of the form $\sigma(X_{t-h}) dW_t$ i.e one single delay in front the driving Brownian motion with a particular form of degeneracy on $\sigma$.  In contrast, in \cite{takeuchi2007a}, there is no delay in the martingale term and the author is only able to prove that a form of strong hypoellipticity similar to our Assumption \ref{hormander-condition} implies the regularity of laws. Furthermore, the analyzed SDEs have drift terms with various levels of generality: \cite{BelMo91} treats the case where the drift vanishes and \cite{stoica1998} treats an explicit form. When considering these aspects only, the classes of SDEs just described are all particular cases of Eq. \eqref{eq:sde_def} and the conditions which yield regularity are more restrictive than the paper's main theorem. This paper is the only piece of literature where a form of weak hypoellipticiy as in Assumption \ref{hormander-condition} is shown to be sufficient for the regularity of laws. This feature allows us treat the classical example  of Langevin Equation with delay where the noise is introduced to the totality of the system via the drift term. 

Let us also point out \cite{ST83} where non-Markovian SDEs is considered, and in particular \cite[Eq. 9.1 p.365]{ST83} which exhibits a path-dependence via a kernel. The forms of path dependence in \cite{ST83,takeuchi2007a} are beyond the scope of this paper and are understood as the limiting case when there is a continuum of delays.

\medskip

{\bf Contributions of the paper: } We give a criterion for smoothness of the density of $X_T$, which is expressed as a spanning condition, unlike the previous contributions for delayed SDEs. This form is closer in spirit to H\"ormander's original theorem \cite{hormander1967} and is ready-to-use.

Furthermore, the three following aspects of the proof are interesting in their own right.
\begin{itemize}
 \item We enhance Brownian trajectories with their delays into a rough path, showing it has all the desired features. Thus, we are able to follow the general lines of Hairer's streamlined proof \cite{Hairer11}, where rough path theory allows to bypass some of the technicalities of Malliavin calculus.
 \item We exhibit a new phenomenon, which we loosely qualify as noise flowing from the past: Delays manifest in the spanning condition through semi-brackets.  
 \item The Lie algebras appearing in the spanning condition are always larger or equal than their Markovian counterparts.  Thus, thanks to the noise flowing from the past, the regularity criterion has better chances of being satisfied.
\end{itemize}

\medskip

\subsection{Setting and main statement}
Let $\alpha$ be a real satisfying $\third < \alpha < \half$ and let $\Cc^\alpha$ be the Banach space of $\alpha$-H\"older continuous functions and for a given $a$, $\Cc^\alpha_{a}$ is the subset of $\Cc^\alpha$ that contains paths equal to $a$ at time $0$. For a vector field
$$
\begin{array}{cccc}
F: & [0,T]\times \R^{d\times N} & \to     & \R^d\\
   & (t, x_0, x_1, \dots, x_{N-1})  & \mapsto & F\left( t, x_0, \dots, x_{N-1} \right)
\end{array}
$$
and a path $(t,\xb)\in [0,T]\times \Cc^\a$ we define the partial derivatives of the functional $F$ as the elements given by:
\begin{align}
\label{def:partial-derivatives}
\pa_i F(t,\xb)&=(\pa_{x_i}F)(t,\xb_t,\ldots,\xb_{t-h_{N-1}}) \in \Lc( \R^d, \R^d ) \ ,\\
\pa_t F(t,\xb)&=(\pa_{t}F)(t,\xb_t,\ldots,\xb_{t-h_{N-1}})\in \Lc(\R, \R^d) \ .
\end{align}
Note that for all $i=0,\ldots N-1$ these partial derivatives measure the sensitivity of the vector field $F$ with respect to the $i$-th delay. The action of $\pa_i F$ on a vector a vector $v \in \R^d$ will be denoted $\pa_i F \cdot v$.

The delay case is also the setting of \cite{nnt2008} where the authors prove well-posedness for SDEs driven by rough paths and in particular for fractional Brownian motion with Hurst parameter $H>\third$. Also, given the right integration framework, which will be given in the preliminaries, we shall see that there is global existence and uniqueness of solutions under the following analytic assumptions.

\begin{assumption}[The analytic assumptions]
\label{assumptions:analytic}
The family of functions 
$$ V_k:\R_+\times \R^{d\times N}\to \R^d $$ 
are smooth with bounded derivatives at all order and satisfy for all $k=1\ldots m$,
\begin{align}
\label{eq:lip_product}
\pa_0 V_k \cdot V_k \mbox{ is uniformly Lipschitz} . 
\end{align}
\end{assumption}

\medskip

In the main Theorem \ref{thm:hormander}, we shall give a criterion in the form of spanning conditions, which is a geometric assumption. We now define the analog of H\"ormander's condition for delayed diffusions. 

\begin{definition}
\label{hormander-spaces}
1) We introduce first the Lie brackets of the vector fields with respect to the end-point of $X$:
\bea\label{dupire-bracket}
[U,V] =\pa_0 U \cdot V-\pa_0 V \cdot U
\eea
where $\pa_0$ stands for the derivative defined at \eqref{def:partial-derivatives}.

2) Given the SDE \eqref{eq:sde_def}, we define sets of vector fields which span the Lie algebra generated by the $V_k$:
$$
\Vc_0 := \left\{(s,\xb)\to V_k(s,\xb): k=1,\cdots, m \right\} \mbox{ and }
$$
$$
\Vc_{j+1} := \Vc_j \cup \left\{[U,V_k]: U \in \Vc_j, k=1,\cdots, m \right\}.
$$
3) We also define extensions of these sets by the contribution of $V_0$ and semi-brackets:
\begin{align*}
  \overline \Vc_j & := \Vc_j \\
& \bigcup\left\{(s,\xb)\to[F,V_0](s,\xb)+\pa_t F(s,\xb) +\sum_{i=1}^{N-1} \pa_i F(s,\xb) \cdot V_0(s-h_i,\xb):F\in \Vc_{j-1} \right\} \\
& \bigcup\left\{ (s,\xb)\to\pa_i F(s,\xb) \cdot V_k(s-h_i,\xb): F\in\Vc_{j-1},\  k=1\cdots m,\ i=1,\dots, N-1 \right\}.
\end{align*}
\end{definition}

\begin{rmk}
i) The bracket $[F,V_k]$, $\pa_t F$ and $\pa_i F$ are well defined for all $F\in \Vc_j$ and $k\geq 0$ since both $F(t,\xb)$ and $V_k(t,\xb)$ can be expressed as smooth functions of $(t,\xb_t,\ldots \xb_{t-h_{N-1}})$.

ii) Note that for $j>0$, the set $ \Vc_j$  is smaller than its Markovian counter-part that also contains the brackets with $V_0$. The bracket with $V_0$ is introduced at $\overline \Vc_j$ and we are able to infer regularity results for diffusions such as the Langevin equation with delay (treated in Section \ref{subsection:examples}).

iii) The fundamental difference between $ \Vc_j$ and $\overline \Vc_j$ is the fact that the elements of the first are functions of $(t,\xb_t,\ldots \xb_{t-h_{N-1}})$ but not the elements of the latter. 
\end{rmk}

Notice that in the non-Markovian case the functionals $V_k$ are necessarily depending on $t$ in a peculiar manner as time plays a special role. Thus, unlike \cite{Hairer11} for example, we made the choice of {\it not} treating the time variable as an additional dimension and adjust the brackets with respect to $V_0$ by adding the time derivative. However, the estimate on the drift part does not allow us to give separate contributions for each of the different terms in the sum 
$$(s,\xb)\to[F,V_0](s,\xb)+\pa_t F(s,\xb) +\sum_{i=1}^{N-1} \pa_i F(s,\xb) \cdot V_0(s-h_i,\xb).$$
We can only rely on the contribution of the sum to produce smoothness.

\begin{assumption}[The geometric assumption - H\"ormander's hypoellipticity condition]
\label{hormander-condition}
We assume {\it either} of the following hypotheses:
\begin{enumerate}
    \item Strong hypoellipticity: $\exists j_0$ such that
\beaa 
\inf_{|\eta|=1}\inf_{\xb\in \Cc^\a_{X_0}} \sup_{F\in \Vc_{j_0}} |\eta F(T,\xb)|>0 \ .
\eeaa

    \item Weak hypoellipticity : $\exists j_0$ such that
\beaa 
\inf_{|\eta|=1}\inf_{\xb\in  \Cc^\a_{X_0}} \sup_{F\in \overline \Vc_{j_0} } |\eta F(T,\xb)|>0 \ .
\eeaa
    \item The bounded case : The process $(X_t)_{t\in [0,T]}$ is uniformly bounded by a deterministic constant and there exists $j_0\geq 0$ such that  for all $\xb\in\Cc^\alpha_{X_0}$ and for all $|\eta|=1$ the pointwise H\"ormander condition holds
\bea \label{pointwise-hormander}
\sup_{F\in \overline \Vc_{j_0} } |\eta F(T,\xb)|>0 \ .
\eea
\end{enumerate}
\end{assumption}
 
\medskip

Clearly, the stronger condition (1) of Assumption \ref{hormander-condition} implies the weaker condition (2), while the condition (3) above is easier to check in the case of bounded processes. We are now ready to state the main result of the paper. 

\begin{thm}
\label{thm:hormander}
Let $X_T$ be the marginal of the solution to the SDE \eqref{eq:sde_def}. If the analytic assumption  \ref{assumptions:analytic} is satisfied, as well as either of the geometric assumptions in \ref{hormander-condition}, then $X_T$ has a smooth density with respect to the Lebesgue measure.
\end{thm}
We give the proof of this theorem only at Subsection \ref{subsection-proof-hormander} after handling all the prerequisites. For now, let us sketch the idea of proof.
\begin{proof}[Idea of proof]
As shown by Malliavin calculus (see \cite{Nualart2006}), smoothness of the law of $X_T$ is implied by a control of the Malliavin matrix. The Malliavin matrix $\Mc_{0,T} \in M_d(\R)$ is a random variable introduced in Section \ref{section:malliavin_argument}, and needs to have an inverse with enough moments. This is implied by the statement that for all $\eta \in \R^d$, $\langle \eta, \Mc_{[0,T]} \eta \rangle$ is small with small probability.

The relationship with Lie brackets is as follows: $\langle \eta, \Mc_{[0,T]} \eta \rangle$ is bounded from below by processes expressed in term of the vector fields $\left( V_j \right)_{1 \leq j \leq m}$. In turn, these are themselves driven by Brownian motion integrated against Lie brackets. The integrands are again processes which are driven by Brownian motion driven by higher order Lie brackets and so on.

The idea is that if the $\langle \eta, \Mc_{[0,T]} \eta \rangle$ is small, then vector fields have to be small (along a certain direction), which themselves cannot be small unless higher order Lie brackets are small (along a certain direction). By induction on the order of Lie brackets, this is excluded by the H\"ormander spanning condition. Therefore, we are looking at an event of small probability.

Here comes in the very nice idea of injecting some rough path technology in order to obtain quantitative estimates. This has been implemented by Hairer in his streamlined proof \cite{Hairer11} of the classical H\"ormander criterion. The required statement is Norris' Lemma, which says informally that if a process remains small in absolute value, while being an integral with respect to Brownian motion, then the integrand {\it has to be } small. This statement is crucial in order to successfully prove the induction step in H\"ormander's criterion and can be given a quantitative form using the technology of rough paths. 

So far, we have given an account of the content of \cite{Hairer11}, written in the Markovian setting. Along these lines, our paper develops the correct framework of rough paths for delay equations, establishes a Norris' Lemma in this framework and then proceeds to prove invertibility of the Malliavin matrix.
\end{proof}

\medskip

\subsection{Examples}
\label{subsection:examples}
Let us illustrate the scope of the main theorem.

\subsubsection{Uniformly elliptic diffusions}
This is the standard example where we assume that there exists $\ve>0$ such that for the order of symmetric matrices we have
$$VV^*\geq \ve id.$$
Note that under this assumption the uniform spanning condition holds for $j_0=0$ and we obtain the smoothness of $X_T$. However this result is not new. Indeed, it is shown in \cite[Theorem 8.3]{ST83}, \cite{KS1984} and \cite{bally2016pathwise}, in a general path-dependent framework, that this condition implies the smoothness of $X_T$.
\subsubsection{Langevin Equation with Delay}
Now, here is an example where the usual Hörmander criterion extends as is to the setting of delays. Consider the diffusion in $\R^2$
\begin{align*}
dp_t & =V_0(p_t,q_t,p_{t-h},q_{t-h})dt + V_1(p_t,q_t,p_{t-h},q_{t-h}) \circ dW_t\\
dq_t & =p_t dt \ .
\end{align*}
with $V_1$ uniformly elliptic. By checking the spanning condition, one realizes that $\left( \Vc_j; j \geq 0 \right)$ is stationary from the index $j=0$ and for all $j \geq 0$:
\begin{align*}
  & \Vc_0=\left\{\left( \begin{matrix}
 V_1 \\
  0
 \end{matrix}\right)\right\} = \Vc_j \ .
\end{align*}
We compute $\overline \Vc_0$:
 \begin{align*}
& \left[\left( \begin{matrix}
 V_1 \\
  0
 \end{matrix}\right),\left( \begin{matrix}
 V_0 (p_s,q_s,p_{s-h},q_{s-h})\\
  p_s
 \end{matrix}\right)\right] + \pa_1\left( \begin{matrix}
 V_1 \\
  0
 \end{matrix}\right) \left( \begin{matrix}
 V_0(p_{s-h},q_{s-h},p_{s-2h},q_{s-2h}) \\
  p_{s-h}
 \end{matrix}\right)
 = \left( \begin{matrix}
* \\
  -V_1 \end{matrix} \right) \ .
 \end{align*}
Hence the uniform spanning condition is satisfied and we have the regularity of law of $X_T$. 

Similarly to \cite{takeuchi2007a} our drift term $V_0$ is allowed to be path-dependent. However, unlike \cite{takeuchi2007a} where strong hypoellipticity is assumed, we are able to treat the case of weak hypoellipticity. In particular, our spanning condition exploits the contribution of $V_0$ to the smoothness of laws. 

\subsubsection{Noise flowing from the past and semi-brackets}
Finally, let us give an example exhibiting a new phenomenon. We now consider the diffusion:
\begin{align*}
\left( \begin{matrix}
  p_t \\
  q_t\\
  r_t
 \end{matrix}\right)= 
  \int_0^t\left( \begin{matrix}
  1 \\
  1\\
  -r_{s-h}
 \end{matrix}\right)dW^1_s
+\int_0^t \left( \begin{matrix}
 -p_{s-h} \\
  q_{s-h}\\
 \sqrt{1+ r^2_{s-h}}
 \end{matrix}\right) dW^2_s \ .
\end{align*}
Again, we check the spanning condition. We have:
\beaa
&&\Vc_0= \left\{ \left( \begin{matrix}
  1 \\
  1 \\
  -r_{s-h} \end{matrix}\right),
  \left( \begin{matrix}
-p_{s-h} \\
  q_{s-h}\\
 \sqrt{1+ r^2_{s-h}} \end{matrix}\right)\right\}=\Vc_j \ .
\eeaa
 We compute the semi-brackets $\partial_1 V_2(t) V_1(t-h)$, $\partial_1 V_1(t) V_2(t-h)$ hence finding a subset of $\overline \Vc_0$. We have:
 \beaa
&&
\left\{
\left( \begin{matrix} 1 \\ 1\\ -r_{s-h}
       \end{matrix}\right),
\left( \begin{matrix} -1 \\ 1\\ -\frac{r_{s-h}r_{s-2h}}{\sqrt{1+r_{s-h}^2}}
       \end{matrix}\right),
\left( \begin{matrix} 0 \\ 0\\ -\sqrt{1+r_{s-2h}^2}
 \end{matrix}\right)\right\} \subset \overline \Vc_0 \ .
\eeaa 
Again, the uniform spanning condition is satisfied and $X_s$ has smooth densities for $s>h$. As the previous computation shows, semi-brackets are crucial in this case in order to create regularity.

\subsection{Structure of the paper}
In the Preliminaries of Section \ref{section:preliminaries}, we start by making precise rough path integration against Brownian motion and its delays. This will show that Eq. \eqref{eq:sde_def} is well posed in Stratonovich form with unique solutions, as well as that it is compatible with the Itô setting. In particular, Subsection \ref{subsection:wellposedness} shows how to reformulate our Hörmander criterion starting from the Itô setting.

Section \ref{section:derivatives} defines and collects results on the Malliavin derivative in this non-Markovian context.

Section \ref{section:malliavin_argument} finally proves the main result. We define the classical Malliavin Gram matrix, quickly review how its control yields smoothness and relate it to tangent flows. 

\subsection{Acknowledgements}
R.C. and I.E. would like thank Y. Bruned, L. Coutin and J. Teichmann for fruitful conversations. Research of I.E. is partly supported by the Swiss National Foundation Grant SNF 200021\_153555. 

Both authors thank the referees for their valuable and constructive comments.

\section{Preliminaries}
\label{section:preliminaries}

From now on, $m$ will refer to the number of Brownian motions we will be working with and $d$ is the dimension of the process $X$ we will study. $\{e_j\}_{j=1,\cdots,d}$ is the canonical basis of $\R^d$ and $\{f_k\}_{k=1,\cdots,m}$ the canonical basis of $\R^m$. $M_d(\R)$ denotes the set of $d$ dimensional matrices. 

Throughout the paper $E$ will stand for a finite dimensional vector space. For any $E$, we denote by $\Cc([0,T],E)$ the space of continuous $E$-valued paths. $\Cc^\alpha([0,T],E)$ will denote the subspace of $\alpha$-H\"older continuous functions. We will drop the dependence on $E$ if it is obvious from context. Also, given a path $X:[0,T] \rightarrow E$ and $(s,t) \in [0,T]^2$, we write the increment between $s$ and $t$ as $X_{s,t} = X_t - X_s$.

Let us start by giving a meaning to Equation \eqref{eq:sde_def} and a solid foundation to its treatment.

\subsection{Stochastic integration and rough paths}
\label{subsection:rough_paths}

\subsubsection{Enhancing Brownian motion to a rough path}
Recall that $W$ is an $m$-dimensional Brownian motion. There is no loss of generality in assuming that $W$ is two-sided: $\left( W_t; t\geq 0 \right)$ and $\left( W_{-t};t \geq 0 \right)$ are independent Brownian motions. Taking $W$ to be two-sided will avoid technical boundary effects and delays can be arbitrarily large.

In this section we give statements for any given $h = \left( h_i \right)_{0 \leq i \leq N-1} \in \R_+^{N}$ increasing sequence of delays. The results will be valid upon changing $h$ to another sequence of delays if necessary. We set $E(h) := \R^{m \times N}$ and consider the $E(h)$-valued process:
$$ \Wc_t(h) = \left(\Wc_t^{k,j}\right)_{\left\{{k=1,\ldots, m,\ j=0,\ldots, N-1}\right\}}
        := \left( W_t, \ W_{t-h_1}, \ W_{t-h_2}, \, \dots, W_{t-h_{N-1}} \right), $$
where each component is understood as a vector in $\R^m$.
When $h$ is understood from context, we drop the dependence on $h$ and write $\Wc_t$ instead of $\Wc_t(h)$. A relevant quantity will be the first delayed date before maturity
\begin{align}
\label{def-th}
T_{h}:= T-{h_1}.
\end{align}
The goal of this subsection is to establish that $\Wc$ is a bona-fide rough path against which we can integrate. We now give a lemma concerning the quadratic covariation of the process $\Wc$. 

\begin{lemma}
\label{lemma:zero_cov}
Consider two indices $i,j$, and two reals $r,r'$. For a partition $\Pc$ of $[s,t]$ with mesh size $|\Pc|$ going to zero, we have the limit in $L^2$ and in probability:
$$ \langle W^i_{r+\cdot}, W^j_{r'+\cdot} \rangle_{s,t}
:= \lim_{|\Pc| \rightarrow 0} \sum_{[u,v] \in \Pc} \left( W^i_{r+v} - W^i_{r+u} \right) \left( W^j_{r'+v} - W^j_{r'+v} \right)
 = \delta_{i,j} \delta_{r,r'} (t-s) \ .$$
\end{lemma}
\begin{proof}
The case $r=r'$ is obvious. By symmetry and time shifting, we can assume $r>0$ and $r'=0$. For shorter notations, set
$$ \Cc =\sum_{[u,v] \in \Pc} \left( W^i_v - W^i_u \right) \left( W^j_{r+v} - W^j_{r+u} \right) \ .$$
Upon assuming $2|\Pc|<r$, we have that $\E\left( \Cc \right) = 0$ as the intervals $[u,v], [u+r, v+r]$ are disjoint. Also:
\begin{align*}
    \E\left( \Cc^2 \right)
= & \sum_{[u,v], [w,\ell] \in \Pc} 
    \E\left[ \left( W^i_v - W^i_u \right) \left( W^j_{r+v} - W^j_{r+u} \right) 
             \left( W^i_\ell - W^i_w \right) \left( W^j_{r+\ell} - W^j_{r+w} \right) 
      \right]\\
= & \sum_{[u,v] \in \Pc} (v-u)^2 \E\left( \Nc^2 \right)^2 \stackrel{|\Pc| \rightarrow 0}{\longrightarrow} 0 \ .
\end{align*}
Indeed, in the double sum, the right-most interval among $[u,v], [u+r, v+r], [w,\ell], [w+r, \ell+r]$ does not intersect the others except if $[u,v] = [w,x]$.
\end{proof}

We fix $\a,\theta$ satisfying $1/3< \a<1/2< \theta<2\a$. Recall that an $\alpha$-Holder path $X$ is lifted to a rough by adjoining another path $\X$ which is $2\alpha$-Holder. We recall the following definition from \cite{FH14} of the topology we use:
\begin{definition}
\label{def-alpha-holder}
We say that the pair $(X,\X)$ is an $\alpha$-Holder rough path on a Banach space $E$, if the mappings 
$$X:[0,T]\to E\mbox{ and } \X:[0,T]^2 \to E\otimes E$$
satisfy
\begin{align}
\Vert X\Vert_\a:= \sup_{0\leq s<t \leq T} \frac{|X_{s,t}|}{|t-s|^\a}<\infty,\quad \Vert \X\Vert_{2\a}:= \sup_{0\leq s\neq t\leq T}\frac{|\X_{s,t}|}{|t-s|^{2\a}}<\infty.
\end{align}
In such a case, by abuse of notation, we say that $(X,\X)\in \Cc^\a([0,T],E)$. We write $\vertiii {X }_{\a}:=\Vert X\Vert_{\a}+\sqrt{\Vert\X\Vert_{2\a}}$ and $\Vert X\Vert_{\Cc^\a}:=\Vert X\Vert_{\a}+\Vert X\Vert_{\infty}$.

For $(X,\X)\in \Cc^\a([0,T],E)$, we say that $(X,\X)$ is a geometric rough path, denoted $(X,\X)\in \Cc_g^\a([0,T],E)$, if 
\begin{align}
\label{def-geometric}
Sym (\X_{s,t})= \frac{1}{2}X_{s,t}\otimes X_{s,t},\mbox{ for all }t,s\in [0,T].
\end{align}
\end{definition}

We first need to define the first order iterated integrals of $\Wc$ in order to form the lift $\W$ also known as L\'evy stochastic areas. With $E = E(h)=\R^{m \times N}$, it is an $(E \otimes E)$-valued path and it is given for $s<t$:
\begin{align}
\label{def-W-lift}
 \left(\W_{s,t}^{k_1,i_1,k_2,i_2}\right)_{\left\{\substack{k_1,k_2=1,\ldots, m,\\ i_1,i_2=0,\ldots, N-1}\right\}}= \left( \int_{s}^t W^{k_1}_{s-h_{i_1},r-h_{i_1}} 
                         \circ dW^{k_2}_{r-h_{i_2}}\right)_{\left\{\substack{k_1,k_2=1,\ldots m,\\ i_1,i_2=0,\ldots, N-1}\right\}} \ .
\end{align}
The matter at hand is to give a precise meaning of the above integrals in such a way that the "first order calculus" condition \eqref{def-geometric} holds. We have two possibilities.

The first possibility is to define the iterated integrals as limits in probability of Riemann sums. More precisely, if $\Pc$ is a partition of $[s,t]$ with mesh size $|\Pc| \rightarrow 0$ and $(X,Y)$ is a pair of paths:
$$ \int_{s}^t X_r \circ d Y_r
   = \lim_{|\Pc| \rightarrow 0}
     \sum_{[u,v] \in \Pc}
     \half \left( X_u + X_v \right)
     \left( Y_v - Y_u \right)
.$$
The above limit is well-defined for $\left( X, Y \right) = \left(W^{k_1}_{r-h_{i_1}}, W^{k_2}_{r-h_{i_2}}\right)$ as soon as $h_{i_2} \leq h_{i_1}$. Indeed, the left-centered Riemann sum converges by standard adapted It\^o integration and we can use Lemma \ref{lemma:zero_cov} to pass to the Stratonovich case. In the other case $h_{i_2} > h_{i_1}$, notice that we have the first order calculus rule at the discrete level:
\begin{align}
\label{eq-discrete-first-order}
   \sum_{[u,v] \in \Pc} \half \left( X_u + X_v \right)\left( Y_v - Y_u \right)
 + \sum_{[u,v] \in \Pc} \half \left( Y_u + Y_v \right)\left( X_v - X_u \right)
 = X_t Y_t - X_s Y_s
\end{align}
and as such the first term converges to a limit as soon as the second does. Therefore, \eqref{def-W-lift} is well-defined as limits in probability of Riemann sums and gives a geometric rough path (i.e. satisfying the equality \eqref{def-geometric}), as the first order calculus rule is built-in at the discrete level already.

The second possibility is to invoke an anticipative integration theory such as Skororhod's. In their paper \cite{nualart1988stochastic} Section 4, Nualart and Pardoux form anticipative Riemann sums which are centered "A la Stratonovich", prove that they converge and relate them to Skorohod's integral. In any case limits in probability of Riemann sums and anticipative Stratonovich integrals "A la Nualart-Pardoux" coincide. See \cite[Exercise 5.17]{FH14} as well as \cite{ocone1989generalized}. 

\subsubsection{Integration with respect to $\Wc$}
Thanks to this paragraph, for systems controlled by delays, we will give a proper meaning to the integration in Eq. \eqref{eq:sde_def}. 

\begin{definition}
\label{def-controlled}
We say that $Y\in \Cc^\a ([0,T], E)$ is controlled by $\Wc$ on a Banach space $E$ if there exists $Y'\in \Cc^\a ([0,T], \Lc(E(h),E))$ such that $R_Y$ defined by 
\begin{align}
\label{controlled}
Y_{s,t}= Y'_s \Wc_{s,t} + R_Y(s,t),\mbox{ for all }0\leq s\leq t\leq T,
\end{align}
satisfies $\Vert R_Y \Vert_{2\alpha}<\infty$. We denote the space of controlled rough paths by $$(Y,Y')\in \Dc_\Wc^{2\alpha} ([0,T], E)$$
and the norm in this space by 
$$\Vert Y,Y'\Vert_{\Wc,2\alpha}:=|Y_0|+\Vert Y'\Vert_{\Cc^\alpha}+ \Vert R_Y\Vert_{2\alpha}.$$
\end{definition}

We recall the following integration theorem due to Gubinelli \cite[Theorem 4.10]{FH14}:
\begin{thm}
\label{thm:continuity_integration}
For every Banach space $E$, for all $(Y,Y')\in \Dc_{\Wc}^{2\a} ([0,T], \Lc(E(h),E))$ the controlled integration mapping
$$
\begin{array}{cccc}
\Dc_{\Wc}^{2\a} \left( [0,T], \Lc(E(h),E) \right) &
\longrightarrow &
\Dc_{\Wc}^{2\a} \left( [0,T], E \right)\\
(Y,Y') & \mapsto & \left(\int_0^\cdot Yd\Wc_t ,Y\right)
\end{array}
$$
where 
\begin{align*}
   \int_s^t Y(r)  d\Wc_r
 &= \sum_{ \substack{ 1 \leq k \leq m \\
                     0 \leq i \leq N-1 } }
    \int_s^t Y^{k,i}(r)  d\Wc^{k,i}_r
\\
&:= \sum_{ \substack{ 1 \leq k \leq m \\
                     0 \leq i \leq N-1 } }
   \lim_{|\Pc|\to 0}\sum_{[u,v]\in\Pc}\left( Y^{k,i}_u \Wc_{u,v}^{k,i}+ \sum_{ \substack{ 1 \leq k_2 \leq m \\
                     0 \leq i_2 \leq N-1 } }
((Y^{k,i})'_u)^{k_2,i_2} \W_{u,v}^{k,i,k_2,i_2}\right)
\end{align*}
is continuous and bounded with bound 
$$   \left\lVert \int_0^\cdot Yd\Wc_t ,Y \right\rVert_{\Wc,2\alpha} 
\leq C\left( \Vert Y\Vert_{\Cc^\a}
  +  \Vert \Wc\Vert_\a\Vert R_Y\Vert_{2\a}+\Vert \W\Vert_{2\a} \Vert Y'\Vert_{\Cc^\a}
     \right) \ .$$
\end{thm}

\subsubsection{Roughness of \texorpdfstring{$\Wc$}{Wc}}
For the convenience of the reader, we recall the concept of roughness for rough paths as in \cite[Definition 6.7]{FH14}. The goal of this subsection is to prove roughness for $\Wc$. This will be crucial in order to use the so-called Norris lemma, a quantitative version of the Doob-Meyer decomposition. 

\begin{definition}
A path $X:[0,T]\to E$ is called $\theta$-Holder rough on scale $\ve_0>0$ and on the interval $[0,T]$ if there exists $L>0$ such that for any linear form $\varphi \in E^*$, $s\in[0,T]$ and $\ve\in(0,\ve_0)$, there exists $t\in [0,T]$ such that 
$$|t-s|\leq \ve \ ,\mbox{ and } \ |\varphi X_{s,t}|\geq L\ve^\theta|\varphi|.$$
The largest such $L$ is called the modulus of $\theta$-Holder roughness of $X$.
\end{definition}

\begin{lemma}
\label{lemma:roughness}
We can choose a version of the Brownian motion such that $\Wc$ is $\theta$-Holder rough at scale $\frac{T}{2}$ on $[0,T]$.
\end{lemma}
\begin{proof}
The proof of roughness is exactly the same as the proof of \cite[Proposition (6.11)]{FH14}. The only ingredient that is missing is the small ball estimate for $\Wc$, which we now prove.

Set $\Delta(h):= \min_{0 \leq i \leq N-2} \left|h_i - h_{i+1}\right|$ with the convention that it is infinity when $N=1$. We shall prove that there are constants $c, C > 0$ such that for all $\varepsilon>0$, $\delta>0$ and $\varphi = \left( \varphi_{i,k} \right) \in E^*$:
\begin{align}
\label{eq:small_ball}
     \P\left( \sup_{0 \leq t-s \leq \delta} \left| \varphi \Wc_{s,t} \right| < \varepsilon \right)
\leq C \exp\left( -\frac{c |\varphi|^2 \left( \delta \wedge \Delta(h) \right) }{\varepsilon^2} \right) \ ,
\end{align}
where $|\varphi|$ is the Euclidean norm. This estimate is sufficient to replace \cite[Eq. (6.11) p. 91]{FH14} so that all the arguments carry verbatim. To prove Eq. \eqref{eq:small_ball}, we start by using the translation invariance and symmetry of Brownian motion increments:
\begin{align*}
  & \P\left( \sup_{0 \leq t-s \leq \delta} \left| \varphi \Wc_{s,t} \right| < \varepsilon \right)\\
= & \P\left( \sup_{0 \leq t-s \leq \delta}
             \left| \sum_{i=1}^{N-1} \sum_{k=1}^m \varphi_{i,k} W_{s-h_i, t-h_i}^k\right| < \varepsilon \right)\\
= & \P\left( \sup_{0 \leq t \leq \delta}
             \left| \sum_{i=1}^{N-1} \sum_{k=1}^m \varphi_{i,k} W_{h_i, h_i-t}^k\right| < \varepsilon \right)\\
\leq & \P\left( \sup_{0 \leq t \leq \delta \wedge \Delta(h)}
             \left| \sum_{i=1}^{N-1} \sum_{k=1}^m \varphi_{i,k} W_{h_i, h_i-t}^k\right| < \varepsilon \right) \ .\\
\end{align*}
Now notice that $W$ being a two sided Brownian motion, the family of processes 
$$\left( t \mapsto W_{h_i, h_i-t}^k ; 0 \leq t <\Delta(h) \right)_{i,k}$$
are independent Brownian motions as we have increments over disjoint intervals when changing $i$ and independent Brownian motions when changing $k$. As such by packaging them into a single Brownian motion $B$ and then invoking the standard small balls estimates, there exists constants $c, C$ such that:
\begin{align*}
       \P\left( \sup_{0 \leq t-s \leq \delta} \left| \varphi \Wc_{s,t} \right| < \varepsilon \right)
\leq & \P\left( \sup_{0 \leq t \leq \delta \wedge \Delta(h)}
                \left| B_t \right| < \frac{\varepsilon}{|\varphi|}\right) \\
\leq & C \exp\left( -\frac{c |\varphi|^2 \left( \delta \wedge \Delta(h) \right) }{\varepsilon^2} \right) \ .
\end{align*}
\end{proof}
\begin{rmk}\label{rmk-extension}
The use of the two sided Brownian motion instead of the Brownian motion allows us to cancel the boundary effects in Lemma \ref{lemma:roughness}. However there is small price to pay here. In order to make the formulas work in the sequel, for all $k=0,\ldots ,m$, we extend $V_k$ to negative times:
$$V_k(s,\xb)=0,\mbox{ for all }s<0.$$
Note that this extension is continuous on $(-\infty,0)$ and on $(0,\infty)$. We will need to be careful with the possible discontinuity of higher derivatives at time $0$, as mentioned in the upcoming Remark \ref{rmk:handlingDiscontinuities}. 
\end{rmk}

\subsection{Well-posedness of the SDE and Itô formula}
\label{subsection:wellposedness}
For completeness, we show that it is possible to reformulate the SDE in an It\^o form as the vector fields are adapted. To do so, one should define the It\^o lift $\W^{\textrm{It\^o}}$ from $\W$ by taking into account quadratic variations:
$$ \W_{s,t}^{k_1, i_1, k_2, i_2} = \left( \W^{\textrm{It\^o}}_{s,t} \right)^{k_1, i_1, k_2, i_2}
                                 + \half  \langle W^{i_1}_{\cdot-h_{k_1}}, W^{i_2}_{\cdot-h_{k_2}} \rangle_{s,t} \ .$$
The covariation of Brownian motion against its own delay is zero which is known as absence of autocorrelation. This was already formalized in Lemma \ref{lemma:zero_cov}. It is classical to see that rough integration against adapted processes and with the It\^o lift $\W^{\textrm{It\^o}}$ coincides with the usual adapted stochastic integration (see for e.g. \cite[Proposition 5.1]{FH14}). As such Eq. \eqref{eq:sde_def} is readily reformulated as an It\^o integral:
\begin{align}
\label{eq:sde_def_ito}
X_t & = X_0 + \int_0^t \left( V_0(r,X) + \half \sum_{k=1}^m  \partial_0 V_k(r,X) \cdot V_k(r,X) \right)dr + \sum_{k=1}^m \int_0^t V_k(r,X) dW^k_r \\
\nonumber
    & = X_0 + \sum_{k=0}^m \int_0^t \widetilde{V_k}(r,X) dW^k_r \ ,
\end{align}
where 
\begin{align*}
\widetilde V_0(r,X)&=V_0(r,X) +  \half \sum_{k=1}^m \partial_0 V_k(r,X) \cdot V_k(r,X) \mbox{ and,} \\
\widetilde V_i(r,X)&=V_i(r,X)\mbox{ for all } 1 \leq i \leq m. 
\end{align*}
Notice that $X$ is a semi-martingale although the integrands in the SDE are not necessarily semi-martingales. The reader more familiar with the It\^o  framework rather than rough paths can establish well-posedness of Eq. \eqref{eq:sde_def_ito}. Indeed, the vector fields $\widetilde{V}_k$ are Lipschitz continuous in the variable $\xb$ for every fixed $t$, uniformly. We have existence and uniqueness of strong solutions to the Equation \eqref{eq:sde_def} via a standard implementation of the Picard iteration scheme, only in the function space $\Cc\left( [0,T], \R^d \right)$ (see the more general theorem 4.6 in \cite{LiShi2001}). Via standard arguments in this framework, solutions are global with $\E\left( \sup_{0 \leq s \leq T} |X_s|^p \right)^{\frac{1}{p}} < \infty$.

Leaving the Itô framework, let us show that Eq. \eqref{eq:sde_def} is well-posed within the theory of rough differential equations \cite[Chapter 8]{FH14}.

\begin{proposition}[Rough forms for the SDE]
There exists a unique process 
$$
\left( X_t ; 0 \leq t \leq T \right)
$$ which solves both the SDE \eqref{eq:sde_def_ito}, formulated in term of Itô integrals and the RDE:
\begin{align}
\label{eq:sde_def_rough}
X_t & = X_0 + \sum_{k=0}^m \int_0^t V_k(t,X) d\Wc^{k,0}_r \ .
\end{align}
As such it is both a semi-martingale and a controlled rough path satisfying 
$$\E\left[ \Vert X,\{V_k(\cdot, X)\}_{k=1}^m\Vert_{\Wc,2\alpha}^p\right] < \infty , \mbox{ for all }p\geq 1.$$
\end{proposition}
\begin{proof}
In \cite[Theorem 8.4]{FH14}, it is explained that RDEs with smooth coefficients have locally unique solutions. Moreover, solutions are global in time thanks to \cite[Proposition 2]{lejay2012global}, which gives boundedness under weaker conditions than the Analytical Assumptions \ref{assumptions:analytic} (\cite[Hypothesis 1]{lejay2012global}).
\end{proof}

We now want some sort of Itô formula for processes of the form $$F(t,X_t,X_{t-h_1},\ldots, X_{t-h_{N-1}}).$$ It is easy to see that this process has no reason to be a semi-martingale and as such, we can only give a rough integral formulation of the Itô lemma. Also after formal computations one notices that the Gubinelli derivative of this process are not controlled by $\Wc$ and they cannot be integrated with respect to $\Wc$. However, they are controlled by $\{W_{t-(h_j+h_i)}\}_{0\leq i\leq j\leq N-1}$. 

 Thus we define the family of double delays
$$\overline h:=\{h_{j_1}+h_{j_2}:j_1,j_2=0\dots N-1\}$$
and choose a family of index $J\subset \{0,\ldots, N-1\}^2$ with minimal cardinality such that $\overline h=\{h_{j_1}+h_{j_2}:j=(j_1,j_2)\in J  \}$.
By the construction at Subsection \ref{subsection:rough_paths}, we obtain a rough path $\Wc_t( \overline{h} ):=\{W_{t-(h_{j_1}+h_{j_2})}\}_{j\in J}$ along with its first order iterated integral $\W(\overline{h})$. The following holds:

\begin{proposition}[It\^o formula]
\label{proposition:equi}
Let $F: \R_+ \times E(h) \rightarrow \R^d$ be a smooth function of time and $\left( X_t, X_{t-h_1}, \dots, X_{t-{N-1}} \right)$. The path $t \mapsto F(t,X)$ is controlled by $\Wc(h)$ as the following control equation holds, for all $0\leq s\leq t\leq T$ with $s\notin \{h_j:j=0\ldots N-1\}$, we have:
\begin{align}
\label{eq:control_V_j}
F(t,X)-F(s,X) = & 
\sum_{ \substack{ 1 \leq k \leq m \\
                  0 \leq i \leq N-1 }} 
\pa_i F (s,X) \cdot V_k (s-h_i,X) W^{k}_{s-h_i,t-h_i}
+R_{F}(s,t) \ ,
\end{align} 
with $\E[\|R_{F}\|_{2\alpha}^p]<\infty$ for all $p>0$. 
Moreover, the Gubinelli derivatives of $t \mapsto F(t,X)$ are controlled by $\Wc(\overline h)$ as for all $i=0,\ldots N-1$, $k=1\ldots,m $ and $0\leq s\leq t\leq T$ such that $s\notin \overline h$ we have:
\begin{align}
\label{eq:control_V_j_prime}
&\pa_{i} F (t,X)V_{k} (t-h_{i},X)-\pa_{i} F (s,X)V_{k} (s-h_{i},X)\\
\nonumber
&\quad=\sum_{ \substack{ 1 \leq l \leq m \\
                         0 \leq j \leq N-1
                       }
            }
        \pa_{i,j}^2 F(s,X) \cdot V_{k}(s-h_{i},X) \cdot V_l(s-h_j,X) W_{s-h_j,t-h_j}^{l}\\
\nonumber
&\quad+ \sum_{ \substack{ 1 \leq l \leq m \\
                  0 \leq j \leq N-1 }} \pa_{i}F(s,X) \cdot \left[ \pa_i V_{k}(s-h_{i},X) \cdot V_l(s-h_i-h_{j},X) \right] W_{s-h_{i}-h_j,t-h_{i}-h_j}^{l}\\
\nonumber
&\quad + R_{F,i,k}(t,s)
\end{align}

Finally, we have the following rough integrals against $\left( \Wc( \overline{h} ), \W( \overline{h} ) \right)$:
\begin{align}
\label{eq:ito_vj_rough}
F(t,X)-F(s,X) = & 
\int_s^t 
\sum_{ \substack{ 1 \leq k \leq m \\
                  0 \leq i \leq N-1 }} 
\pa_i F (r,X) \cdot V_k (r-h_i,X) d\Wc^{k,i}_r+\int_s^t \pa_t F (r,X) dr \ .
\end{align}
\end{proposition}
\begin{proof}
Both control equations hold by virtue of a Taylor expansion and the use of Eq. \eqref{eq:sde_def_rough}.
Eq. \eqref{eq:ito_vj_rough} is obtained by invoking the rough path It\^o's formula given in \cite[Theorem 7.6]{FH14}. The first hypothesis required is that $F$ and its Gubinelli derivative are controlled by $\left( \Wc(\overline h), \W(\overline h) \right)$, which is a consequence of the two control equations \eqref{eq:control_V_j} and \eqref{eq:control_V_j_prime}. The second hypothesis \cite[p.100 Eq.(7.8)]{FH14} requires the computation of a Taylor expansion at order $2$. Notice that since we use geometric rough paths, the bracket $\left[ \Wc \right]$ defined in \cite[Definition 5.5]{FH14} is zero.
\end{proof}
\begin{rmk}
\label{rmk:handlingDiscontinuities}
Following Remark \ref{rmk-extension} and the choice of extension for $V_k$, the equalities \eqref{eq:control_V_j} and \eqref{eq:control_V_j_prime} for Gubinelli derivatives are only stated on open intervals between their successive delays. 
\end{rmk}

\section{The Malliavin derivative}
\label{section:derivatives}

In the context of performing probabilistic constructions and estimating densities, one needs to be able to differentiate with respect to Brownian trajectories. This contribution of Malliavin brought functional analysis to probability.

\subsection{Derivatives}

Let us start by the general notion of Fr\'echet derivative of a functional (see \cite[p.128]{kriegl-michor1997}):
\begin{definition}[\bf Fr\'echet derivative ]
\label{frechet-derivative}
For $H$ any vector subspace of $\Cc := \Cc\left( [0,T], \R \right)$, the continuous functions from $[0,T]$ to $\R$, and $F: \Cc \to \R^d$, we define $DF(\xb)(\varphi)\in\R^d$, the Fr\'echet derivative of $F$ at point $\xb \in \Cc$ and direction $\varphi\in H$, as the limit 
\bea
DF(\xb)(\varphi):=\lim_{\varepsilon \to 0 }\frac{F(\xb+\varepsilon \varphi)-F(\xb)}{\varepsilon}\in\R^d \ ,
\eea
when it exists. The limit needs to hold, uniformly in $\varphi$ belonging to the unit ball. For ease of notation, if $F$ takes functions in $\Cc^d$ as input, i.e from $[0,T]$ to $\R^d$, then for all $\varphi \in \Cc^d$, we also define the Fr\'echet derivative matrix $DF(\xb)(\varphi) \in M_d(\R)$ as the matrix whose $j$-th column is 
$$ \lim_{\varepsilon\to 0 }\frac{F(s,\xb+\varepsilon(\varphi)_j)-F(s,\xb)}{\varepsilon} \in \R^d$$
where $(\varphi)_j\in \Cc$ is the path of the $j$-th column of $\varphi$. The functional $F$ is said to be Fréchet differentiable if the Fréchet derivative exists and is a bounded linear operator. The operator norm is the supremum norm.
\end{definition}

In the particular case of Brownian motion, we obtain the Malliavin derivative.  In order to define the Malliavin derivative we introduce the Cameron-Martin space
$$
H:=\left\{ \varphi \in L^0([0,T]; \R^m): \ \varphi' \in L^2([0,T]; \R^m), \ \varphi(0) = 0 \right\}.
$$
Let $F$ be an $\R^d$-valued smooth functional on $\Cc\left( [0,T], \R^m \right)$, evaluated on the Brownian motion $W$. The Malliavin derivative of $F$ applied to $\varphi \in H$ is defined as the Fréchet differential of $F$:
$$ \Dc F \cdot \varphi := \lim_{\ve\to 0}\frac{F(W+\ve \varphi)-F(W)}{\ve} \in \R^d  \ .$$
The iterated Malliavin derivatives $\Dc^j$ are defined in the same fashion from higher order Fréchet differentials. As such, $\Dc^j$ is seen as acting on random variables which are smooth functions of $W$. Then $\Dc^j$ is extended to the domain $\D^{j,p}$ in $L^p(\Omega), \ p \geq 1$, with respect to the norm:
$$ \left\| F \right\|_{j,p} = \left[ \E \left(|F|^p\right) + \sum_{k=1}^j \E\left( \left| \Dc^k F \right|^p_{H^{\otimes k}} \right)\right]^{\frac{1}{p}}$$
Moreover, we write:
$$ \D^{j, \infty} := \cap_{p \geq 1} \D^{j, p} \ .$$
For further details we refer to \cite[Section 1.2]{Nualart2006}.

A standard notation is to represent the Mallavin derivative as an element in the Cameron-Martin space 
$$ \Dc F = \left( \Dc_t^1 F, \dots, \Dc_t^m F \right)_{0 \leq t \leq T}$$
and write:
$$ \Dc F \cdot h = \sum_{j=1}^m \int \Dc_t^j F \ \langle f_j^*, h'(t) \rangle dt \ ,$$
where $\left( f^*_j \right)_{1 \leq j \leq n}$ is the basis dual to the standard basis $\left( f_j \right)_{1 \leq j \leq n}$ of $\R^m$ and $\langle \cdot, \cdot \rangle$ is the duality bracket.

{\it Morally}, at time $t$ and for $j=1,\cdots, m$,  the operator $\Dc_t^j$ is given by:
$$ \Dc^j_t F := \lim_{h\to 0}\frac{F(W+\ve \mathds{1}_{[t,T]} f_j)-F(W)}{\ve}\in\R^d  \ .$$
We denote by $\Dc_t F \in M_{d,m}\left( \R \right)$ the matrix whose $j$-th column is $\Dc^j_t F$. The following proposition sums up the properties of the Malliavin derivative in our context.

\begin{proposition}[Kusuoka-Stroock \cite{KS1984}]
For all $t \leq r$, the random variable $X_r$ belongs to the space $\D^{1, \infty}$. Moreover, for all $j=1\ldots m$, the Malliavin derivative $\Dc_r^j X_t$ of the random variable $X_t$ satisfies:
\begin{align}
\label{malliavin-x}
\Dc_r^j X_t & = {V}_j (r,X) + \sum_{k=0}^m \int_{r}^t  D\widetilde{V}_k (s,X)\left(\Dc_r^j X_{\cdot} \right) dW_s^k.
\end{align}
One also has $\Dc_r^j X_t=0$ if $t<r$, as $X_r$ is adapted. In matrix notation one has 
\begin{align*}
\Dc_r X_t   & = {V}(r,X) + \sum_{k=0}^m \int_r^t D\widetilde{V}_k (s,X)\left(\Dc_r   X_{\cdot}\right) dW_s^k
\end{align*}
where $V(r,X)$ is the matrix whose columns are $V_j(r,X)$ for $j=1\ldots m$. 
\end{proposition}
\begin{proof}[Pointers to the proof]
This is essentially \cite[Lemma (2.9)]{KS1984}. Note that the latter reference uses Itô's formulation. Thus we start with the equation \eqref{eq:sde_def_ito} and see that it satisfies the assumptions in \cite[Lemma (2.9)]{KS1984}. 

Due to the analytical assumptions \ref{assumptions:analytic} the functions $\widetilde{V}_k$ admits Fréchet derivatives at all order and for all $h\in H$ the Malliavin derivatives $( \Dc X_s(h))_{s\in [0,T]}$ solves the SDE 
\begin{align}
\Dc^j X_t \cdot h = \int_0^t {V}_j(r,X)h'_r dr+\sum_{k=0}^m \int_0^t D\widetilde{V}_k(s,X)(\Dc^j X_\cdot \cdot h ) dW_s^k \ .
\end{align}
Additionally, as proven by Kusuoka and Stroock, the mapping $\Dc^j X_t$ is a Hilbert-Schmidt operator on $H$ hence the existence of $\{\Dc_r^j X_t\}_{r\in[0,T]}\in L^2([0,T]:\R^d)$ satisfying
$$
\Dc^j X_t(h)=\int_0^t \left[ \Dc_r^j X_t \right] h'_r dr \mbox{ for all }h\in H.
$$
Note that the equality  
$$
D\widetilde{V}_k(s,\xb)(\xb' ) = \sum_{i=0}^{N-1} \pa_i \widetilde V_k (s,\xb) \cdot \xb'_{s-h_i}
$$
implies thanks to Fubini
\begin{align*}
\int_0^t \Dc^j_r X_t h'_r dr&=\int_0^t {V}_j(r,X)h'_r dr+\sum_{k=0}^m\sum_{i=0}^{N-1} \int_0^t\int_0^t \pa_i\widetilde{V}_k(s,X) \cdot \Dc^j_r X_{s-h_i} dW_s^k h'_rdr\\
\notag&=\int_0^t {V}_j(r,X)h'_r dr + \sum_{k=0}^m \int_0^t\int_0^t D\widetilde{V}_k(s,X) (\Dc^j_r X_\cdot) dW_s^k h'_rdr
\end{align*}
which implies by identification \eqref{malliavin-x}.
\end{proof}

\subsection{Factorization of the Malliavin derivative}

The main result of this section concerns a factorization of the Malliavin derivative.

\begin{proposition}
\label{proposition:malliavin_derivative_expr}
Define the family of processes $\left( J_{r,t} ; 0 \leq r \leq t \leq T \right)$ as the solution to the SDE taking values in $M_d(\R)$:
\begin{align}
\label{eq:flow1}
J_{r,t} & = \id + \sum_{k=0}^m \int_{r}^t  D\widetilde{V}_k (s,X)\left( J_{r, \cdot} \right) dW_s^k ,\ \mbox{ for } r \leq t, \\
\nonumber 
J_{r,t} & = 0, \mbox{ for } r > t  \ .
\end{align}
Here $\id$ stands for the identity matrix. Then, for $0\leq r \leq t\leq T$, the tangent process and the Malliavin derivative satisfy 
\bea\label{equation-tangent-malliavin}
 J_{r, t}  \times  V\left( r,X \right)
 = \Dc_r \left( X_t \right) 
 \eea
 where on the left-hand side, the product denotes a matrix product.
\end{proposition}
\begin{proof}
Inspecting equations \eqref{malliavin-x} and \eqref{eq:flow1}, we recognize the same stochastic differential equation with a different initial condition. The starting condition $ V_j\left( r, X\right)$ in equation \eqref{malliavin-x} is replaced by the constant $f_j$ in the matrix equation \eqref{eq:flow1}. The equation \eqref{eq:flow1} is linear in the original condition and one can multiply by the constant $  V_j\left( r,X \right)$ to identify the Malliavin derivative and the multiplied flow. Hence we obtain the result.
\end{proof}

\begin{rmk}
\label{rmk:tangent_markov}
i) In the Markovian setting, let $X_{\cdot}^{t,x}$ be the solution to \eqref{eq:sde_def} such that $X_{t}^{t,x}=x$. By uniqueness of the solution, there exists flow maps 
$$\left( \Phi_{t,s}: \R^d \rightarrow \R^d \right)_{0 \leq t \leq s \leq T}$$
such that $\Phi_{t,s}(x) = X_s^{t,x}$. It is well-known that $\Phi$ are in fact flows of diffeomorphisms. We recommend the works of Kunita for example (\cite{K84} and \cite[Chapter 4]{kunita1997}). The tangent process is a process of invertible linear maps $J_{t,T}(x): \R^d \rightarrow \R^d$ obtained via:
$$ \forall H \in \R^d, \  J_{t,T}(x) \cdot H 
                       := \lim_{\varepsilon \rightarrow 0} \frac{X_T^{t, x+\varepsilon H} - X_T^{t,x}}{\varepsilon}
                        = d \Phi_{t,T} (x) \cdot H \ .$$
Here $d$ stands for the usual differential. Moreover, we have:
\begin{align*}
J_{t,T}(x) = & \id + \sum_{k=0}^m \int_t^T 
               \left[ \frac{\partial \widetilde{V_k}}{\partial x}\left( r, X_r^{t,x} \right) \cdot J_{t,r}(x) \right] dW_r^k \ ,
\end{align*}
which is virtually the same equation as \eqref{eq:flow1} only that Fréchet derivatives of the vector fields have replaced the usual derivative. Here $J_{r,t}$ can be understood as the sensitivity of $X_t$ to a variation of the point $X_r$. It is also well-known that in the Markovian framework the equality 
$$
J_{t, T}(X_r)  \times  V\left( r,X_r\right)
 = \Dc_t \left( X_T \right)
 $$
 holds.
 
 ii) Note that we do not endow $\{J_{r,t}\}$ with the classical interpretation of derivative of the flow here. This family is only defined as the solution of \eqref{eq:flow1}. 
\end{rmk}

\subsection{Analysis \texorpdfstring{on $[T_h, T]$}{near T} }

Classically, the Lie bracket in the H\"ormander's condition appears through the evolution of $J^{-1}_{t,s}$ together with the vector fields $V_k(t,X)$. As such it will be crucial to understand the evolution of $J^{-1}_{t,s}$. However, in the non-Markovian framework the matrix-valued process $J_{t,s}$ might fail to be invertible at all times. 

However thanks to the delay structure, a perturbation of $X$ at time $t \in [T_h, T]$ will affect $X_T$ through only $\pa_{0} V_k (s,X)$. This is seen in the simplification of Eq. \ref{eq:flow1} on the interval $[T_h,T]$. We treat it in the following proposition.

\begin{proposition}
\label{proposition:tangent_flow_sde} 
$J$ satisfies the following SDE, for $T_h\leq t \leq s \leq T$
\begin{align}
\label{eq:tangent_flow_sde}
J_{t,s} &= \id +\sum_{k=0}^m \int_{t}^s  \pa_0 \widetilde{V_k} (r,X) \cdot J_{t,r} dW_r^k \ ,
\end{align}
where  we take the convention that $J_{t,r}=0$ of $r<t$. Moreover, $\{J_{t,r}\}_{T_h \leq t\leq r \leq T}$ exists for all time and satisfies the following moment bounds:
$$ \forall p \geq 1, \, \E\left( \sup_{T_h \leq t \leq r \leq T} \lVert J_{t,r} \rVert^p \right) < \infty$$
\end{proposition}
We also give the following proposition allowing us to differentiate $J_{t,T}$ in $t$. 
\begin{proposition}
For all $T_h<s<t<T$ we have the following relation,
\label{j-before-T}
\begin{equation}
\label{jt-rde}
 J_{t,T}-J_{s,T}=-\sum_{k=0}^m \int_{s}^t J_{r,T} \cdot \pa_{0} V_k (r,X)  d \Wc_r^{k,0} \ 
\end{equation}
where the integral is understood as a rough integral with respect to $\Wc$. 
\end{proposition}
\begin{proof}
By direct computation we see that the SDE \eqref{eq:tangent_flow_sde} can be written on $T_h \leq s \leq t \leq T$ as the rough differential equation (RDE):
$$ J_{s,t} = \id + \int_s^t \sum_{k=0}^m \pa_{0} V_k(r,X) \cdot J_{s,r} d\Wc_t^{k,0} \ ,$$
and because we are dealing with a linear RDE, $\left( J_{s,T} ; T_h \leq s \leq T \right)$ remains invertible and we have the splitting:
\begin{align}\label{splitting}
J_{s,T} = J_{T_h, T} J_{T_h, s}^{-1} \ .
\end{align}
Applying the chain rule for $dJ^{-1} = - J^{-1} dJ J^{-1}$, we have:
\begin{align*}
    J_{t,T} - J_{s,T} 
= & J_{T_h, T} J_{T_h, t}^{-1} - J_{T_h, T} J_{T_h, s}^{-1}\\
= & -J_{T_h, T} \int_s^t J_{T_h, r}^{-1} \cdot \left[ \sum_{k=0}^m \pa_0 V_k(r,X) \cdot J_{T_h, r}^{-1} \right] \cdot J_{T_h, r}^{-1}d \Wc^{k,0}_r \\
= & -J_{T_h, T} \sum_{k=0}^m \int_s^t J_{T_h, r}^{-1} \cdot \pa_0 V_k(r,X) d \Wc^{k,0}_r \\
= & -\sum_{k=0}^m \int_{s}^t J_{r,T} \cdot \pa_0 V_k (r,X)d\Wc_r^{k,0} \ .
\end{align*}
\end{proof}

\begin{rmk}
The splitting property \eqref{splitting} is one of the main limitations of this paper. This property gives the invertibility of $\{J_{s,T}\}_{s\in[T_h,T]}$ and its regularity in $s$. This property does not hold for $s\leq T_h$ since there would be an extra noise coming from the delays. Additionally, when $V_k$ is a general path-dependent functional there is no obvious way to obtain the invertibility of $J_{s,T}$ and its regularity in $s$. 
\end{rmk}

\section{Malliavin's argument: smoothing by Gaussian noise}
\label{section:malliavin_argument}

The gist of Malliavin's argument is that the random variable $X_T$ is a complicated function of the Brownian motion $W$. Provided that such a map is smooth enough, and because Gaussian noise is smooth, one expects $X_T$ to have a smooth density. The quantity that encodes this dependence is the Malliavin matrix $\Mc_{0,T} \in M_d(\R)$ which is defined as:
\begin{align}
\label{eq:malliavin_matrix}
 \Mc_{0,T}:= & \int_0^T \Dc_s X_T \left( \Dc_s X_T \right)^* ds=\left\{ \int_0^T  \sum_{k=1}^m \Dc_s^k \left(X_T\right)^i \Dc_s^k \left(X_T\right)^j ds\right\}_{i,j}\ .
\end{align}
It is morally a Gram matrix or a covariance matrix of the sensitivities of $X_T$ to the Brownian motion $W$. The norm of its inverse will control the smoothness of the map $W \mapsto X_T$. Let $\eta \in \R^d$ such that $\left| \eta \right|_{\R^d}=1$. 
 Thanks to Proposition \ref{proposition:malliavin_derivative_expr} we have:
\begin{align*}
    \langle \eta, \Mc_{0,T} \eta \rangle_{\R^d}
= & \int_0^T \langle \eta,  \Dc_s X_T \left( \Dc_s X_T \right)^* \eta \rangle_{\R^d} ds\\
\geq & \int_{T_h}^T \left| V(s,X)^* J_{s,T}^* \eta \right|_{\R^m}^2 ds\\
= & \sum_{j=1}^m \int_{T_h}^T \left| \eta^* J_{s,T} V_j(s,X) \right|_{\R^d}^2 ds \ .
\end{align*}

\begin{rmk}
In the Markovian case, it is very convenient to introduce the reduced Malliavin matrix $\Cc_{0,T}$ such that $\Mc_{0,T} = J_{0,T} \Cc_{0,T} J_{0,T}^*$. In that case, tangent processes have the multiplicative property 
$$ J_{s,T} = J_{0,T} J_{0,s}^{-1} \ ,$$
and one obtains:
\begin{align}
\label{eq:malliavin_matrix_reduced}
 \Cc_{0,T} := & \int_0^T J_{0, s}^{-1}
                         V\left( s, X\right)
                         \left( V\left( s,X\right) \right)^* 
                         \left( J_{0, s}^{-1} \right)^*
                          ds \ ,
\end{align}
which is an adapted process. This classical trick allows to use It\^o calculus to study the matrix $\Cc_{0,T}$ and relate its evolution to iterated Lie brackets, thus to the H\"ormander's condition. See the general guidelines of Theorem 4.5 in \cite{Hairer11}.

However, in our setting, such an approach is not possible because the infinitesimal flow property (Eq. \eqref{eq:flow1}) takes a more complicated form. It is a priori not obvious to find a reduced Malliavin matrix which is the integral of an adapted process. This is the reason why we perform an analysis only on the segment $[T_h, T]$.
\end{rmk}

\subsection{The evolution of \texorpdfstring{$Z_{F}$}{the Malliavin matrix} and its derivatives}
In this subsection, we fix a functional of time and $\left( X_t, X_{t-h_1}, \dots,X_{t-h_{N-1}} \right)$ denoted by $F:\R^+\times (\R^{d})^N\to \R^d$ and compute the expansion as a rough integral of $\left\{\eta^* J_{t,T}F(t,X) \right\}_{t\in[T_h,T]}$ on the path $\Wc$. For notational simplicity we define 
$$Z_F(t):= \eta^* J_{t,T}F(t,X) \ .$$
The underlying assumption is that $F$ is smooth and all of its derivatives at any order are bounded.

 Recall that the rough path $\left( \Wc(\overline h), \W(\overline h) \right)$ is the lift of $W$ taken with the family of double delays $\overline h$ and defined at subsection \ref{subsection:wellposedness}. We also mentioned at Remark \ref{rmk-extension} that the functionals $V_k$ have discontinuities at time $h_i$. In order to avoid problems due to this lack regularity and to be able to use the Norris' lemma we define 
$$T_{\overline h}:=\sup\{ \overline h \cap [0,T)\}\vee T_h \in (0,T).$$
Note that on the interval the analysis above concerning the Malliavin derivative holds. We also have the following lemma where all the integrands are free of discontinuities on $[T_{\overline h},T]$. 
\begin{lemma}
\label{lem-zF}
For all $T_{\overline h}<s\leq t<T$, we have
\begin{align}\label{eq:zF}
Z_{F}(t)-Z_{F}(s) & = \sum_{k=1}^m \left[ \int_s^t Z_{[F, V_k]}(r) d\Wc_r^{k,0}+ \sum_{i=1}^N  \int_s^t Z_{\pa_i F (\cdot,X) \cdot V_k (\cdot-h_i,X)}(r)d\Wc_r^{k,i} \right] \\
&+\int_s^t Z_{\{\pa_t F(\cdot ,X)+ [F,V_0]+\sum_{i=1}^{N-1} \pa_i F(\cdot,X) \cdot V_0(\cdot-h_i,X)\}}(r) dr\notag
\end{align}
where all the integrands are controlled by $\Wc(\overline h)$ and the integrals are rough integrals of $ \Wc(\overline h)$. 
Additionally, 
\begin{align}
\E\left[\Vert Z_{F}\Vert^p_{\a, [T_{\overline h},T]}\right]<\infty,\mbox{ for all }p\geq 2,
\end{align}
and the remainder $R_{F}$ defined by 
\begin{align*}
R_{F}(s,t):= Z_{F}(t)-Z_{F}(s) -\sum_{k=1}^m  \left[ Z_{[F, V_k]}(s) \Wc_{s,t}^{k,0}+ \sum_{i=1}^{N-1} Z_{\pa_i F (\cdot,X) \cdot V_k (\cdot-h_i,X)}(s) \Wc_{s,t}^{k,i} \right] \ , 
\end{align*}
for $T_{\overline h}<s\leq t\leq T $ and  $s\notin \overline h$  satisfies
$$\E\left[\Vert R_{F}\Vert_{2\a,[T_{\overline h},T]}^p\right]<\infty,  \mbox{ for all } p\geq 2.$$
\end{lemma}
\begin{proof}
Apply the Leibniz rule on the product $Z_F(s) = \eta J_{s,T} F(s)$, and then use the rough integral expansions for $F$ (Proposition \ref{proposition:equi}) and $J_{\cdot, T}$ (Proposition \ref{j-before-T}).
\end{proof}

Note that we can apply Lemma \ref{lemma:roughness} for the rough path $\Wc(\overline h)$ and obtain that this path is $\theta$-Holder rough. We can apply the Norris' Lemma in \cite[Theorem 3.1]{HP13} in the following form. 

\begin{lemma}[Norris lemma]
\label{norris}
There exist constants $p,r>0$ such that for all $(A,A')\in \Dc_{\Wc(\overline h)}^{2\alpha} ([0,T], V)$ and $B$ $\alpha$-H\"older continuous, the path defined for $t,s\in [T_{\overline h},T]$ by 
\begin{align}\label{eq:norris}
Z(t) -Z(s)= \int_s^t B_r dr+\sum_{k=1}^m \sum_{i=0}^{N-1}  \int_s^t A^{k,i}_r d\Wc_r^{k,i}  
\end{align}
satisfies
$$\Vert A\Vert_{\infty,  [T_{\overline h},T]} + \Vert B\Vert_{\infty, [T_{\overline h},T]} \leq C \Rc^p  \Vert Z\Vert_{\infty, [T_{\overline h},T]}^r$$
where the constant $C$ depends only on $T,\{h_i\}$ and $m$ and 
$$ \Rc:=1+L_\theta^{-1}( \Wc(\overline h))+\vertiii{  \Wc(\overline h)}_\a+\Vert A,A'\Vert_{ \Wc(\overline h),2\alpha}+\Vert B\Vert_{\Cc^\alpha,[T_{\overline h},T]}.$$
\end{lemma}
The choice of $T_{\overline h}$ is mainly motivated by the fact that $\Vert B\Vert_{\Cc^\alpha,I}$ might become infinite if the interval $I$ contains an element of $\overline h$. 

For notational simplicity we define the key quantity for all $n\in \N$
\begin{align}
\Rc_n := & 1+L_\theta^{-1}( \Wc(\overline h))+\vertiii{  \Wc(\overline h)}_\a + 
           \sum_{{F\in \bar \Vc_n}} \left[ \|Z_{F} \|_{\Cc^\alpha,[T_{\overline h},T]} + \|R_{F}\|_{2\alpha,[T_{\overline h},T]} \right]
\end{align}
which satisfies $\E[\Rc^p_n]<\infty$ for all $p>0$ and $n\in\N$. Recall the definition of $j_0$ in the assumption \ref{hormander-condition}. It is the rank such that $\overline{\Vc}_{j_0}$ has the uniform spanning condition.

\begin{lemma}
\label{lem-control-derivative1}
Fix $j_0 \in \N$. There exist deterministic constants $p_0,q_0,C>0$ depending on $j_0$, $T$ and $\{h_i\}$, such that for all $F\in \overline \Vc_{j_0}$, we have 
\begin{align}\label{eq:zfj0}
|Z_F|_{\infty, [ T_{\overline h},T]} \leq  C\Rc_{j_0}^{p_0}\langle \eta, \Mc_{0,T} \eta \rangle^{q_0} .
\end{align}
\end{lemma}
\begin{proof}
We reason by induction over the index $j_0$.  

For initial step $j_0 = 0$, we start by the fact that there exists a constant $C_{h,T}$ such that for all $j=1\ldots m$
\begin{align}
\label{eq-holder-z1}
\|Z_{V_j} \|_{\infty, [T_{\overline h},T]}  \leq C_{h,T}   \langle \eta, \Mc_{0,T} \eta \rangle^{\frac{\alpha}{2\alpha+1}} \|Z_{V_j} \|^{\frac{1}{2\alpha+1}}_{\alpha,[T_{\overline h},T]} .
\end{align}
To prove that fact, simply repeat the interpolation inequality argument as in the proof of Lemma 5 in \cite{HP13} on $[T_{\overline h},T]$ and obtain that 
\begin{align*}
\sup_{s\in [T_{\overline h},T]} \left| Z_{V_j} (s) \right| 
\leq C_{h,T}  \Vert Z_{V_j} \Vert^{\frac{2\alpha}{2\alpha+1}}_{L^2([T_{\overline h},T])} \|Z_{V_j} \|^{\frac{1}{2\alpha+1}}_{\alpha,[T_{\overline h},T]}
\end{align*}
We finish the proof of Eq. \eqref{eq-holder-z1} with the obvious inequalities
$$\Vert Z_{V_j} \Vert^2_{L^2([T_{\overline h},T])}\leq  \langle \eta, \Mc_{0,T} \eta \rangle.$$
Finally, Eq. \eqref{eq-holder-z1} implies \eqref{eq:zfj0} because $\|Z_{V_j} \|_{\alpha,[T_{\overline h},T]} \| \leq \Rc_0$.

Now for the induction step, we assume that the result holds true for $j_0$. Consider $F \in \overline \Vc_{j_0+1}$. Due to the definition of the brackets at \eqref{dupire-bracket}, there exists $G\in \Vc_{j_0}$ such that $G$ is a function of the form $G:\R^+\times (\R^{d})^N\to \R^d$ and $Z_F$ is a Gubinelli derivative of $Z_G$ or $Z_F$ is the absolutely continuous part in the decomposition of $Z_G$. We apply the Norris Lemma \ref{norris} to Equation \eqref{eq:zF}  for $Z_G$ on  $[T_{\overline h},T]$ to have the existence of $C_1$, $p_1$ and $q_1$ such that 
$$|Z_F|_{\infty, [ T_{\overline h},T]} \leq  C_1 \Rc_{j_0}^{p_1} \sup_{G \in \Vc_{j_0}}\ |Z_G|_{\infty, [ T_{\overline h},T]}^{q_1} \ .$$
This implies by induction hypothesis that there are $p_0$ and $q_0$ such that
$$|Z_F|_{\infty, [ T_{\overline h},T]} \leq  C_1\Rc_{j_0+1}^{p_0}\langle \eta, \Mc_{0,T} \eta \rangle^{q_0}. $$
The fact that this inequality is in particular true for $F\in \Vc_{j_0+1}$ is what we need to iterate. 
\end{proof}
\subsection{Proof of Theorem \ref{thm:hormander}}\label{subsection-proof-hormander}
We first prove the theorem under the Assumption \ref{hormander-condition}, condition (1) or (2). It is classical that $ \E\left[| \Mc_{0,T}^{-1}|^p\right]<\infty$ for all $p\geq 2$ is a sufficient condition for the existence of smooth densities for $X_T$(see for example \cite[Theorem 2.1.4]{nualart1988stochastic}).  As shown in \cite[Lemma 4.7]{Hairer11}, this latter statement is itself implied by the existence for all $p \in \N$ of a constant $C_p$ such that:
\begin{align}
\label{eq:sup_Malliavin}
\sup_{|\eta| = 1} \P\left( \langle \eta, \Mc_{0,T} \eta \rangle \leq \varepsilon \right) \leq C_p \varepsilon^p
\end{align}
We now use the inequality \eqref{eq:zfj0} at time $T$ and obtain 
\begin{align*}
\inf_{|\eta|=1}\inf_{\xb\in \O} \sup_{F\in \overline \Vc_{j_0}} |\eta^* F(T,\xb)|\leq C\Rc_{j_0}^{p_0}\langle \eta, \Mc_{0,T} \eta \rangle^{r_0}. 
\end{align*}
Due to the H\"ormander condition in assumption \ref{hormander-condition} the left hand side is a positive deterministic constant that we denote $\delta>0$. We obtain 
$$\langle \eta, \Mc_{0,T} \eta \rangle\geq\frac{\delta^{1/r_0}}{(C\Rc_{j_0}^{p_0})^{1/r_0}}.$$
Using the integrability of $\Rc_{j_0}$ we easily obtain \eqref{eq:sup_Malliavin}. 

\begin{rmk}[Special case of bounded diffusion]
Note that the classical H\"ormander theorem requires a pointwise spanning condition. This is due to the fact that in the Markovian case the derivative of the flow $J_{t,T}$ is invertible for all $t\in[0,T]$ and the spanning condition is only required at the initial point of the diffusion. We do not have any hope of obtaining this invertibility. Thus we are only able to reason at time $T$ and check a spanning condition at the random variable $(T,X_T)$ via our uniform condition \ref{hormander-condition}. 

Note that if we know a priori that the process is bounded we can still have a more pointwise statement of the H\"ormander condition.  In order to state this result we formulated condition (3) of Assumption \ref{hormander-condition}.
\end{rmk}

We now prove the theorem under condition (3). Denote $C$ the constant bounding $X_t$.
$j_0$ is finite and the functions $F\in \overline \Vc_{j_0}$ are continuous on a finite dimensional space. Thus there exists $\eta^*$ with $|\eta^*|=1$ and $\xb^*\in \Cc^\alpha$ with such that 
\beaa
\inf_{|\eta|=1}\inf_{\substack{\xb\in \O\\|\xb|_\infty \leq C}} \sup_{F\in \overline \Vc_{j_0}} |\eta F(T,\xb)| = \sup_{F\in \overline \Vc_{j_0}} |\eta^* F(T,\xb^*)| 
\eeaa
We now use the H\"ormander condition \ref{pointwise-hormander} at the point $\xb^*$ to obtain that the existence of $F^*\in \overline \Vc_{j_0}$ such that $ |\eta^* F^*(T,\xb^*)| =\d>0$. Similarly to the beginning of this section we obtain that  
$$\langle \eta, \Mc_{0,T} \eta \rangle\geq\frac{\delta  ^{1/r_0}}{(C\Rc_{j_0}^{p_0})^{1/r_0}}.$$
and finish the proof.

\bibliographystyle{halpha}
\bibliography{RegularityHormander}

\end{document}